\documentclass[12pt, a4paper]{article}

\usepackage{lscape}
\usepackage{amsfonts}
\usepackage{amssymb,amsthm,amsmath,color,mathrsfs}
\usepackage[T1]{fontenc}
\usepackage{lmodern}
\usepackage{hyperref}
\usepackage{tikz-cd}

\newtheorem{Th}{Theorem}[section]

\newtheorem{Lem}[Th]{Lemma}
\newtheorem{Prop}[Th]{Proposition}
\newtheorem{Res}[Th]{Result}

\theoremstyle{definition}
\newtheorem{Def}[Th]{Definition}
\newtheorem{Not}[Th]{Note}

\newtheorem{Ex}[Th]{Example}

\newcommand{\R}{\mathbb R}

\newcommand{\N}{\mathbb N}
\newcommand{\C}{\mathbb C}
\newcommand{\K}{\mathbb K}
\newcommand{\X}{\mathcal X}
\newcommand{\com}[2]{\mathscr C_{#2}(#1)}

\begin{document}

\title{Some Special Sets in an Exponential Vector Space}

\author{Priti Sharma\footnote{Bangabasi College, University of Calcutta, 19, Rajkumar Chakraborty Sarani, Kolkata-700009, INDIA, e-mail : mspriti23@gmail.com} and Sandip Jana\footnote{Department of Pure Mathematics, University of Calcutta, 35, Ballygunge 
Circular Road, Kolkata-700019, INDIA, e-mail : sjpm@caluniv.ac.in}}
\date{}
\maketitle

\maketitle

\begin{abstract}
In this paper, we have studied `absorbing' and `balanced' sets in an Exponential Vector Space (\emph{evs} in short) over the field $\K$ of real or complex. These sets play pivotal role to describe several aspects of a topological evs. We have characterised a local base at  the additive identity in terms of balanced and absorbing sets in a topological evs over the field $\K$. Also, we have found a sufficient condition under which an evs can be topologised to form a topological evs. Next, we have introduced the concept of `bounded sets' in a topological evs over the field $\K$ and characterised them with the help of balanced sets. Also we have shown that compactness implies boundedness of a set in a topological evs. In the last section we have introduced the concept of `radial' evs which characterises an evs over the field $\K$ up to order-isomorphism. Also, we have shown that every topological evs is radial. Further, it has been shown that ``the usual subspace topology is the finest topology with respect to which $[0,\infty)$ forms a topological evs over the field $\K$''.
\end{abstract}

AMS Subject Classification : 46A99, 06F99.

 Key words : topological exponential vector space, order-isomorphism, absorbing set, balanced set, bounded set, radial evs.

\section{Introduction}

Exponential vector space is a new algebraic structure consisting of a semigroup, a scalar multiplication and a compatible partial order which can be thought of as an algebraic axiomatisation of hyperspace in topology; in fact, the hyperspace $\com{\X}{}$ consisting of all non-empty compact subsets of a Hausd\"{o}rff topological vector space $\X$ was the motivating example for the introduction of this new structure. It was introduced by S. Ganguly et al. in \cite{C(X)} with the name ``\emph{quasi-vector space}'' (in short ``\emph{qvs}''). In $\com{\X}{}$ we can find the following properties if addition and scalar multiplication are defined by:  $A+B:=\{a+b:a\in A,b\in B\}$ and $\alpha A:=\{\alpha a:a\in A\}$ where, $A,B\in\com{\X}{}$ and $\alpha$ belongs to the scalar field of $\X$.\\
(i) $\com{\X}{}$ is closed under addition and scalar multiplication.\\
(ii) if $A\subseteq B$ then $\alpha A\subseteq \alpha B$, for any scalar $\alpha$ (not necessarily non-negative).\\
(iii) $(\alpha+\beta)A\subseteq\alpha A+\beta A$, for any scalar $\alpha,\beta$ and any $A\in\com{\X}{}$; equality holds iff $A$ is a singleton set.\\
(iv) Singletons are the only invertible elements with respect to the aforesaid addition, $\{\theta\}$ acting as the identity (if $\theta$ is the identity of $\X$).

All these facts were considered to formulate the axioms of quasi-vector space. Although $\com{\X}{}$ is the ``founder'' example of this structure, a large number of examples of such algebraic structure have been found in various branches of mathematics viz. in number system, theory of matrices, collection of positive measures, collection of non-negative functions, lattice theory, collection of subspaces of vector space, to name a few (see \cite{mor},\cite{qvs},\cite{spri},\cite{bal}).

However, we think that this newly defined structure is not merely a generalisation of vector space bearing the name ``\emph{quasi-vector space}'', rather the axioms of this new structure evolve a very rapid growth of the elements of the structure with respect to the partial order and also evoke some sort of positiveness in each element. Meanwhile, a vector space is evolved within this structure and positivity of each element of the new structure is judged with respect to the elements of the vector space generated. All these facts emerge some sort of exponential flavour within the structure. Considering the importance and influence of the partial order which prevails an essence of hyperspace in the entire structure we find that the name ``\emph{exponential vector space}'' should be a suitable nomenclature for the structure previously called ``\emph{quasi-vector space}''. The formal definition of an exponential vector space is as follows :

\begin{Def}\cite{evs} Let $(X,\leq)$ be a partially ordered set, `$+$' be a binary operation on 
$X$ [called \emph{addition}] and `$\cdot$'$:K\times X\longrightarrow X$ be another composition [called \emph{scalar multiplication}, $K$ being a field]. If the operations and partial order satisfy the following axioms then $(X,+,\cdot,\leq)$ is called an \emph{exponential vector space} (in short \emph{evs}) over $K$ [This structure was  initiated with the name  \emph{quasi-vector space} or \emph{qvs} by S. Ganguly et al. in \cite{C(X)} ].
\begin{align*}
A_1&: (X,+)\text{ is a commutative semigroup with identity } \theta\\  
A_2&: x\leq y\ (x,y\in X)\Rightarrow x+z\leq y+z\text{ and } \alpha\cdot x\leq
\alpha\cdot y,\  \forall z\in X,  \forall \alpha\in K\\ 
A_3&:\text{(i)}\ \alpha\cdot(x+y)=\alpha\cdot x+\alpha\cdot y\\ 
&\quad\text{(ii)}\ \ \alpha\cdot(\beta\cdot x)=(\alpha\beta)\cdot x\\ 
&\quad\text{(iii)}\  \ (\alpha+\beta)\cdot x \leq \alpha\cdot x+\beta \cdot x\\
&\quad\text{(iv)}\ \ 1\cdot x=x,\text{ where `1' is the multiplicative identity in }K,\\ 
& \forall\,x,y\in X,\  \forall\,\alpha,\beta\in K \\
A_4&: \alpha\cdot x=\theta \text{ iff }\alpha=0\text{ or }x=\theta\\ 
A_5&: x+(-1)\cdot x=\theta\text{ iff } x\in X_0:=\big\{z\in X:\ y\not\leq z,\,\forall\,y \in X\smallsetminus\{z\}\big \}\\ 
A_6&: \text{For each }x \in X,\, \exists\,p\in X_0\text{ such that }p\leq x.
\end{align*}   
\label{d:evs}\end{Def}

In the above definition, $X_0$ is precisely the set of all minimal elements of the evs $X$ with respect to the partial order of $X$ and forms the maximum vector space (within $X$) over the same field as that of $X$ (\cite{C(X)}). We call this vector space $X_0$ as the `\emph{primitive space}' or `\emph{zero space}' of $X$ and the elements of $X_0$ as `\emph{primitive elements}'. Conversely, given any vector space $V$, an evs $X$ can be constructed (as shown below in example \ref{e:0}) such that $V$ is isomorphic to $X_0$. In this sense, ``exponential vector space'' can be considered as an algebraic ordered extension of vector space. The axiom $ A_3 $(iii) expresses very rapid growth of the non-primitive elements, since $ x\leq \frac{1}{2}x+\frac{1}{2}x,\forall\,x\notin X_0 $; whereas axiom $ A_6 $ demonstrates `positivity' of all elements with respect to primitive elements.

\begin{Ex} \cite{evs}
Let $X:=\big\{(r,a)\in\R\times V:r\geq0,a\in V\big\}$, where $V$ is a vector space over some field $K$. Define operations and partial order on $X$ as follows : for $(r,a),(s,b)\in X$ and $\alpha\in K$,\\
(i) $(r,a)+(s,b):=(r+s,a+b)$;\\ 
(ii) $\alpha(r,a):=(r,\alpha a)$, if $\alpha\neq0$ and $0(r,a):=(0,\theta)$, $\theta$ being the identity in $V$;\\ 
(iii) $(r,a)\leq(s,b)$ iff $r\leq s$ and $a=b$.\\ Then $X$ becomes an exponential vector space over $K$ with the primitive space $\{0\}\times V$ which is evidently isomorphic to $V$.
\label{e:0}\end{Ex}

For topologising an exponential vector space, we need the following concepts:

\begin{Def}\cite{Nach} Let `$\leq$' be a preorder in a topological
space $Z$; the preorder is said to be \emph{closed} if its graph
$G_{\leq}(Z):=\big\{(x,y)\in Z\times Z:x\leq y\big\}$ is closed in $Z\times Z$ (endowed
with the product topology).
\end{Def}

\begin{Th} {\em\cite{Nach}} A partial order `$\leq$' in a topological
space $Z$ will be a closed order iff for any $x,y\in Z$ with
$x\not\leq y$, $\exists$ open neighbourhoods $U,V$ of $x,y$ respectively in
$Z$ such that $(\uparrow U)\cap(\downarrow V)=\emptyset$, where
$\uparrow U:=\{x\in Z:x\geq u \text{ for some }u\in U\}$ and
$\downarrow V:=\{x\in Z:x\leq v \text{ for some }v\in V\}$.
\label{ch1:t:closedorder}\end{Th}

\begin{Def}\cite{evs} An exponential vector space $X$ over the field $\K$ of real or complex numbers is said to be a \emph{topological exponential vector space} if $X$ has a
topological structure with respect to which the addition, scalar
multiplication are continuous and the partial order `$\leq$' is
closed (Here $\K$ is equipped with the usual topology).
\end{Def}

From the definition it follows that if $ X $ is a topological exponential vector space over the field $ \K $, then the primitive space $ X_0 $ becomes a topological vector space, since restriction of a continuous function is continuous. Moreover from the characterisation of closed order (theorem \ref{ch1:t:closedorder}) it follows that every topological evs is Hausd\"{o}rff and hence the primitive space $ X_0 $ becomes a Hausd\"{o}rff topological vector space over $ \K $. We first cite an example which will be useful in the sequel.

\begin{Ex}\cite{mor} Let $X:=[0,\infty)\times V$, where $V$ is a vector space over the field $\K$ of real or complex numbers. Define operations and partial order on $X$ as follows : for $(r,a),(s,b)\in X$ and $\alpha\in\K$,\\
	(i) $(r,a)+(s,b):=(r+s,a+b)$\\ 
	(ii) $\alpha(r,a):=(|\alpha|r,\alpha a)$\\ 
	(iii) $(r,a)\leq(s,b)$ iff $r\leq s$ and $a=b$.\\
	Then $[0,\infty)\times V$ becomes an exponential vector space with the primitive space $\{0\}\times V$ which is clearly isomorphic to $V$.
	
	In this example, if we consider $V$ as a Hausd\"{o}rff topological vector space then $[0,\infty)\times V$ becomes a topological exponential vector space with respect to the product topology, where $[0,\infty)$ is equipped with the subspace topology inherited from the real line $\R$.

	Instead of $V$ if we take the trivial vector space $\{\theta\}$ in the above example, then the resulting topological evs is $[0,\infty)\times\{\theta\}$ which can be clearly identified with the half ray $[0,\infty)$ of the real line. Thus, $[0,\infty)$ forms a topological evs over the field $\K$.
	\label{ch1:e:0inf}\end{Ex}

\begin{Def}\cite{mor} A mapping $f:X\longrightarrow Y$ ($X,Y$ being two exponential vector spaces over a common field $K$) is called an \emph{order-morphism} if\\
	(i) $f(x+y)=f(x)+f(y)$, $\forall\,x,y\in X$\\
	(ii) $f(\alpha x)=\alpha f(x)$, $\forall\,\alpha\in K$, $\forall\,x\in X$\\
	(iii) $x\leq y$ $(x,y\in X)\Rightarrow f(x)\leq f(y)$\\
	(iv) $p\leq q$ $\big(p,q\in f(X)\big)\Rightarrow f^{-1}(p)\subseteq\downarrow f^{-1}(q)$
	and $f^{-1}(q)\subseteq\uparrow f^{-1}(p)$.
	
	A bijective order-morphism is called an \emph{order-isomorphism}.
	\label{mor}\end{Def}

\begin{Def}\cite{evs}
	A property of an evs is called an \textit{evs property} if it remains invariant under order-isomorphism.	
\end{Def}

In this paper, we have studied `absorbing' and `balanced' sets in an evs over the field $\K$ of real or complex. These sets play pivotal role to describe several aspects of a topological evs. We have characterised a local base at  the additive identity in terms of balanced and absorbing sets in a topological evs over the field $\K$. Also, we have found a sufficient condition under which an evs can be topologised to form a topological evs. Next, we have introduced the concept of `bounded sets' in a topological evs over the field $\K$ and characterised them with the help of balanced sets. Also we have shown that compactness implies boundedness of a set in a topological evs. 

In the last section we have introduced the concept of `radial' evs which characterises an evs over the field $\K$ up to order-isomorphism. Also, we have shown that every topological evs is radial. Further, it has been shown that ``the usual subspace topology is the finest topology with respect to which $[0,\infty)$ forms a topological evs over the field $\K$''.

\section{Prerequisites}

\begin{Def} \cite{spri} A subset $Y$ of an exponential vector space $X$ is said to be a \emph{sub exponential vector space} 
(\emph{subevs} in short) if $Y$ itself is an exponential vector space with all the compositions of 
$X$ being restricted to $Y$.
\label{d:subevs}\end{Def} 

\begin{Not} \cite{spri} A subset $Y$ of an exponential vector space $X$ over a field $K$ is a sub exponential vector space iff $Y$ satisfies the following:\\ 
(i) $\alpha x+y\in Y,\ \forall\,\alpha\in K$, $\forall\,x,y \in Y$.\\ 
(ii) $Y_{0}\subseteq X_{0}\bigcap Y$, where $Y_0:=\big\{z\in Y:y\nleq z,\forall\,y\in Y\smallsetminus\{z\}\big\}$\\ 
(iii) $\forall\, y\in Y$, $\exists\,p\in Y_{0}$ such that $p\leq y$. 

If $Y$ is a subevs of $X$ then actually $Y_0=X_0\cap Y$, since for any $Y\subseteq X$
we have $X_0\cap Y\subseteq Y_0$.
\label{n:subevs}\end{Not}

\begin{Def} Let $\{X_i:i\in\Lambda\}$ be an arbitrary family of exponential vector spaces over a common field $K$ and $X:=\displaystyle\prod_{i\in\Lambda}X_i$ be the Cartesian product. Then, $X$ becomes an exponential vector space over $K$ with respect to the following operations and partial order (see section 5 of \cite{qvs}) :

For $x=(x_i)_i, y=(y_i)_i\in X$ and $\alpha\in K$ we define (i) $x+y:=(x_i+y_i)_i$, (ii) $\alpha x:=(\alpha x_i)_i$, (iii) $x\ll y$ iff $x_i\leq y_i$, $\forall\,i\in\Lambda$.\\
Here the notation $x=(x_i)_i\in X$ means that the point $x\in X$ is the map $x:i\mapsto x_i\, (i\in\Lambda)$, where $x_i\in X_i,\ \forall\,i\in\Lambda$. The additive identity of $X$ is given by $\theta=(\theta_i)_i$, $\theta_i$ being the additive identity of $X_i,\ \forall\,i\in\Lambda$. Also the set of all primitive elements of $X$ is given by $X_0=\displaystyle\prod_{i\in\Lambda}[X_i]_0$.
\end{Def}

\begin{Def}\cite{spri} In an evs $X$, the \textit{primitive} of any $x\in X$ is defined as the set\\
\centerline{$P_x:=\{p\in X_0 : p\leq x\}$.} The axiom $A_6$ of the definition \ref{d:evs} ensures that the primitive of each element of an evs is nonempty. Also,  $P_{\alpha x}=\alpha P_x, \forall x\in X, \forall \alpha\in \mathbb{K}.$
\end{Def}

\begin{Ex}\cite{spri} Let $\X$ be a vector space over the field $\K$ of real or complex numbers. Let
$\mathscr L(\X)$ be the set of all linear subspaces of $\X$. We now define $+,\cdot,\leq$ on $\mathscr L(\X)$ as follows :
For $\X_1,\X_2\in\mathscr L(\X)$ and $\alpha\in\K$, \\
(i) $\X_1+\X_2:=$ span$(\X_1\cup \X_2)$, (ii) $\alpha\cdot \X_1:=\X_1$, if $\alpha\neq0$
and $\alpha\cdot \X_1:=\{\theta\}$, if $\alpha =0$
($\theta$ being the additive identity of $\X$), (iii) $\X_1\leq \X_2$ iff 
$\X_1\subseteq \X_2$.

Then $\big(\mathscr L(\X),+,\cdot,\leq\big)$ is a non-topological evs over $\K$.
\label{ch1:e:LX}\end{Ex}

\begin{Ex}\cite{bal}
Let us consider $\mathscr D^2([0,\infty)):=[0,\infty)\times [0,\infty)$.
We define $+,\cdot,\leq$ on $\mathscr D^2([0,\infty))$ as follows :\\
For $(x_1,y_1),(x_2,y_2)\in\mathscr D^2([0,\infty))$ and $\alpha\in\C$ we define\\
(i) $(x_1,y_1)+(x_2,y_2)=(x_1+x_2,y_1+y_2)$\\
(ii) $\alpha\cdot(x_1,y_1)=(|\alpha|x_1,|\alpha|y_1)$\\
(iii) $(x_1,y_1)\leq (x_2,y_2)\Longleftrightarrow$ either $x_1<x_2$ or if
$x_1=x_2$ then $y_1\leq y_2$ [\emph{dictionary order}].
Then, $\big(\mathscr D^2([0,\infty)),+,\cdot,\leq\big)$ forms a non-topological evs over $\C$.
\label{ch1:e:d2inf}\end{Ex}

\section{Basics of Absorbing and Balanced Sets}\markright{Basics of Absorbing and Balanced Sets}

In this section we will define absorbing and balanced sets in an evs and discuss their general properties. 

\begin{Def} Let $X$ be an evs over the field $\K$ of real or complex numbers. A set $A\subseteq X$ is called an \emph{absorbing set}\index{Absorbing set} if for any $x \in X$, $\exists$ a real number $\alpha$ > 0 such that $\mu x \in A,  \forall \mu \in \mathbb{K}$ with $|\mu|\leq \alpha$.
\label{ch4:d:abs}\end{Def}

Clearly, for any absorbing set $A(\subseteq X)$ and $x \in X$, 0.$x \in A$ i.e. $\theta \in A$ [ $\theta$ being the additive identity of $X$ ].

Let $A_i$ be absorbing subsets of $X$, $\forall\, i=1,2,...n$. Then for any $x \in X$, $\exists\  \alpha_i > 0$ such that $\mu x \in A_i$, $\forall \mu \in \mathbb{K}$ with $|\mu|\leq \alpha_i$, $\forall\,i=1,2,...n$. If we set $\alpha:=\displaystyle\min_{1 \leq i \leq n }\alpha_i$, then $\mu x \in \bigcap \limits_{i=1}^{n}A_i$, $\forall\,\mu \in \mathbb{K}$ with $|\mu|\leq\alpha$. Since $x \in X$ is arbitrary, so $\bigcap \limits_{i=1}^{n}A_i$ is also an absorbing subset of $X$. 

Let $A$ be an absorbing subset of $X$ and $A \subseteq B\subseteq X$. So, for any $x \in X, \exists \ \alpha > 0$ such that $\mu x \in A \subseteq B$, $\forall \mu \in \mathbb{K}$ with $|\mu|\leq \alpha$. Hence $B$ is also an absorbing subset of $X$. Now since $A \subseteq \uparrow A$ and $A\subseteq\downarrow A$, so both $\uparrow A \text{ and} \downarrow A$ are absorbing subsets of $X$ whenever $A$ is so.

Again, let $A$ be an absorbing subset of $X$ and $\lambda \in \mathbb{K}^*\equiv\K\smallsetminus\{0\}$. Then for any $x \in X$, $\exists\, \alpha> 0$ such that $\mu x \in A$, $\forall \mu \in \mathbb{K}$ with $|\mu| \leq \alpha$. So, $\beta x \in \lambda A$, $\forall \beta \in \mathbb{K}$ with $|\beta |$ $\leq$ $\alpha | \lambda |$ where, $\alpha | \lambda | > 0$ . Since $x \in X$ is arbitrary, $\lambda A$ is an absorbing subset of $X$.

From the above discussion we can conclude the following basic facts:

\begin{Prop}Let $X$ be an evs over the field $\K$ of real or complex numbers. Then,
\\ (i) $A \subseteq X$ is absorbing $\Longrightarrow \theta \in A$ $[$ $\theta$ being the additive identity of $X$ $]$.
\\(ii) Intersection of a finite collection of absorbing sets is an absorbing set.
\\(iii) Superset of an absorbing set is absorbing and hence union of any collection of absorbing sets is absorbing.
\\(iv) $A \subseteq X$ is absorbing $\Longrightarrow \uparrow A\text{ and} \downarrow A$
are absorbing.
 \\(v) $A \subseteq X$ is absorbing and $\lambda \in \mathbb{K}^* \Longrightarrow \lambda A$ is absorbing.

\label{ch4:p:a1}\end{Prop}

\begin{Def} Let $X$ be an evs over the field $\K$ of real or complex numbers. A set $A\subseteq X$ is called a \emph{balanced set}\index{Balanced set} if for any $\alpha \in \mathbb{K} \text{ with } | \alpha | \leq 1$, $\alpha A \subseteq A.$
\end{Def}

Clearly, for any balanced set $A(\subseteq X)$ and $x \in X$, 0.$x \in A$ i.e. $\theta \in A$.

Let $\{A_\alpha : \alpha \in \Lambda \}$ be a family of balanced sets in an evs $X$, where $\Lambda$ is an index set and $A:=\underset{\alpha \in \Lambda}{\bigcap} A_\alpha$. Let  $\lambda \in \mathbb{K} \text{ with } | \lambda  | \leq 1$ and $x \in A$. Then $x \in A_\alpha$, $\forall \alpha \in \Lambda$ where each $A_\alpha$ is balanced. So, $\lambda x \in A_\alpha$, $\forall \alpha \in \Lambda$ $\Longrightarrow$ $\lambda x\in A$. Therefore, $\lambda A\subseteq A$. Hence $A$ is a balanced subset of $X$. 

Let $B:= \underset{\alpha \in \Lambda}{\bigcup} A_\alpha$ and $x \in B$. So, $\exists\, \alpha \in \Lambda$ such that $x \in A_\alpha.$ As $A_\alpha$ is balanced, so $\lambda x \in A_\alpha \subseteq B$ for any $\lambda\in\K$ with $| \lambda | \leq 1$. Therefore, $\lambda B\subseteq B$. Hence $B$ is a balanced subset of $X$.

Let $\alpha \in \mathbb{K}$ with $| \alpha | \leq 1$ and $A$ be a balanced subset of $X$. Let $\lambda \in \mathbb{K}$ and $y \in \lambda A.$ Then $y = \lambda x$ for some $x \in A.$ Now $A$ is balanced so $\alpha x \in A.$ Therefore, $\alpha y = \alpha (\lambda x) = \lambda (\alpha x) \in \lambda A.$ Thus, $\alpha (\lambda A) \subseteq \lambda A$ which implies that $\lambda A$ is also a balanced subset of $X$.

Now, let $\alpha \in \mathbb{K}$ with $| \alpha | \leq 1$ and $A$ be a balanced subset of $X$. Let $y \in \uparrow A$ be arbitrary. So, $\exists\, x \in A$ such that $x \leq y \Rightarrow \alpha x \leq \alpha y \Rightarrow \alpha y \in \uparrow A\ (\because \alpha x \in \alpha A \subseteq A)$. Thus $\alpha (\uparrow A) \subseteq \uparrow A$, where $| \alpha | \leq 1.$ Hence $\uparrow A$ is a balanced subset of $X$. Similarly, it can be shown that $\downarrow A$ is also a balanced subset of $X$.

From the above discussion we can conclude the following basic facts:

\begin{Prop}Let $X$ be an evs over the field $\K$ of real or complex numbers. Then, 
\\ (i) $A \subseteq X$ is balanced $\Longrightarrow \theta \in A$ $[$ $\theta$ being the additive identity of $X$ $]$.
\\(ii) Intersection of any collection of balanced sets is a balanced set.
\\(iii) Union of any collection of balanced sets is a balanced set.
\\(iv) $A \subseteq X$ is balanced $\Longrightarrow \uparrow A\text{ and} \downarrow A$
are balanced.
 \\(v) $A \subseteq X$ is balanced and $\lambda \in \mathbb{K} \Longrightarrow \lambda A$ is balanced.

\label{ch4:p:b1}\end{Prop}

\section{Effect of Local Base at $\theta$ on the Topology of a Topological evs}\markright{Effect of Local Base at $\theta$ on the Topology of a Topological evs}

In this section, we have characterised a local base at `$ \theta $' $[$ $\theta$ being the additive identity $]$ in terms of balanced and absorbing sets in a topological evs $X$ over the field $\K$ of real or complex numbers. Also, we have found a sufficient condition under which an evs satisfying certain properties can be topologised to form a topological evs over $\K$.

\begin{Res} In a topological evs $X$ over the field $\K$ of real or complex numbers, any neighbourhood of  $\theta$ is an absorbing set and contains a balanced neighbourhood of $\theta$ .
\label{ch4:r:1}\end{Res}

\begin{proof} Let $U$ be any neighbourhood of $\theta$ in $X$ and $x \in X$ be arbitrary. Define a function $\psi : \mathbb{K} \rightarrow X$ by: $\psi (t)= t.x$, $\forall\, t \in \mathbb{K}.$ $X$ being a topological evs, the scalar multiplication `$\cdot$' : $\mathbb{K} \times X \rightarrow X$ is continuous and hence $\psi$ is also continuous. Now, $\psi(0)$ = $\theta$ and $U$ is a neighbourhood of $\theta \Longrightarrow \exists$ a closed disc $\overline{B}(0 , \epsilon):=\{ \alpha \in \mathbb{K} : | \alpha |\leq\epsilon \}$ in $\K$ for some $\epsilon>0$ such that $\psi(\overline{B}(0 , \epsilon))$ $\subseteq U$ i.e. $\overline{B}(0 , \epsilon).x \subseteq U$ i.e. $\lambda x \in U$, $\forall \lambda \in \mathbb{K} \text{ with }$ $| \lambda |$ $\leq$ $\epsilon. $ As $x \in X$ is arbitrary, $U$ is an absorbing set.

Again, $0.\theta = \theta$. So $\exists$ a closed disc $\overline{B}(0 , \epsilon)$ in $\K$ for some $\epsilon>0$ and an open neighbourhood $W$ of $\theta$ in $X$ such that $\overline{B}(0 , \epsilon). W\subseteq U$ $\Longrightarrow$ $V := \underset{| \lambda | \leq \epsilon}{\bigcup} \lambda W$ $\subseteq U$. Now $\epsilon W$ is also a neighbourhood of $\theta$ and $\epsilon W \subseteq V$. Hence $V$ is a neighbourhood of $\theta$. Again, let $\alpha \in \mathbb{K}$ with $| \alpha | \leq 1$ and $x \in V$. Then $x = \lambda x_1$ for some  $x_1 \in W$ and $\lambda \in \mathbb{K}$ with $| \lambda | \leq \epsilon.$ So, $\alpha x = \alpha \lambda x_1$ where, $| \alpha \lambda |$ $\leq$ $| \lambda |$ $\leq$ $\epsilon$ and $x_1 \in W$. Therefore $\alpha x \in$ $\underset{| \lambda | \leq \epsilon}{\bigcup} \lambda W$ = $V.$ Thus, $ \alpha V \subseteq V \text{ for } | \alpha | \leq 1.$ Hence $V$ is a balanced neighbourhood of $\theta$ with $V \subseteq U.$ 
\end{proof}

\begin{Res} In a topological evs $X$ over the field $\K$ of real or complex numbers, for each neighbourhood $U$ of $\theta$,  $\exists$ a neighbourhood $W$ of $\theta$  such that $W+W \subseteq U.$
\label{ch4:r:2} \end{Res}

\begin{proof} In $X$, $\theta + \theta$ = $\theta$ and $+ : X \times X$ $\longrightarrow X$ is a continuous map. So, for any neighbourhood $U$ of $\theta$, $\exists$ neighbourhoods $G$ and $H$ of $\theta$ such that $G + H \subseteq U.$ Put, $W = G \cap H$. Then $W + W$ $\subseteq U$ where, $W$ is also a neighbourhood of $\theta.$ 
\end{proof}

\begin{Res} In a topological evs $X$ over the field $\K$ of real or complex numbers, any open set $G$ is of the form: $G$ $=$ $\underset{x \in G}{\bigcup}(x + U_x)$ where, $U_x$ is a balanced neighbourhood of $\theta$ \big(depending on $x$\big) $\forall x \in G$.
\label{ch4:r:3}\end{Res}

\begin{proof} Let $G$ be an open subset of $X$ and $x$ = $x + \theta \in G$. Since $+ : X \times X$ $\longrightarrow$ $X$ is continuous, $\exists$ open neighbourhoods $V_x$ and $W_x$ of $\theta$ and $x$ respectively such that $W_x + V_x$ $\subseteq$ $G \Longrightarrow$ $x + V_x \subseteq G.$ Now, $V_x$ contains a balanced neighbourhood of $\theta$, say $U_x$ [by result \ref{ch4:r:1}]. So, $x \in x + U_x \subseteq G$. Thus  $\underset{x \in G}{\bigcup}\{x\}$ $\subseteq$ $\underset{x \in G}{\bigcup}(x + U_x)$ $\subseteq G$. i.e. $G$ = $\underset{x \in G}{\bigcup}(x + U_x)$ where, each $U_x$ is a balanced neighbourhood of $\theta$. 
\end{proof}

\begin{Res} In a topological evs $X$ over the field $\K$ of real or complex numbers, if $x \nleq y$ then $\exists$ balanced neighbourhoods $U, V$ of $\theta$ such that $\uparrow(x + U)$ $\cap$ $\downarrow(y + V)=\emptyset$, $\forall x,y\in X$.
\label{ch4:r:4} \end{Res}

\begin{proof} Since $X$ is a topological evs, the partial order `$\leq$' on $X$ is closed. So, for any $x, y \in X$ with $x \nleq y$, $\exists$ open neighbourhoods $G_x$ and $G_y$ of $x$ and $y$ respectively such that $\uparrow G_x$  $\cap$  $\downarrow G_y$ = $\emptyset$ [ in view of theorem \ref{ch1:t:closedorder} ]. Now, $x \in G_x$ and $y \in G_y$ $\Longrightarrow \exists$ balanced neighbourhoods $U$ and $V$ of $\theta$ such that $x + U \subseteq G_x$ and $y + V\subseteq G_y$ [by result \ref{ch4:r:3}]. Clearly, $\uparrow(x + U)$ $\subseteq$ $\uparrow G_x$ and $\downarrow(y + V)$ $\subseteq$ $\downarrow G_y$. Hence $\uparrow(x + U)$  $\cap\downarrow(y + V)=\emptyset$.
\end{proof}

\begin{Res} Let $X$ be a topological evs over the field $\K$ of real or complex numbers and $G$ be an open subset of $X$. Then for any $x \in G$ and $\alpha \in \K\smallsetminus\{0\}$, $\exists$ a balanced neighbourhood $U_x$ of $\theta$ and $\epsilon>0$ such that $B(\alpha,\epsilon).(x+U_x)$ $\subseteq \alpha G$ where, $B(\alpha,\epsilon):=\{\lambda\in\K:|\alpha-\lambda|<\epsilon\}$.
\label{ch4:r:5}\end{Res}

\begin{proof} Let $G$ be an open subset of the topological evs $X$ and $x \in G$. Now, for any fixed $\lambda \in \K\smallsetminus\{0\}$, the map $M_\lambda:$ $X \longrightarrow X$ defined by : $M_\lambda(y)=\lambda y, \forall y\in X,$ is a homeomorphism. So, for any $\alpha \in \K\smallsetminus\{0\},\ \alpha G$ is an open set containing    $\alpha x.$ Since the scalar multiplication $\cdot: \mathbb{K} \times X \rightarrow X$ is continuous, $\exists\, \epsilon>0$ and an open set $U$ containing $x$ such that $B(\alpha,\epsilon).U \subseteq \alpha G$ where, $B(\alpha,\epsilon):=\{\lambda\in\K:|\alpha-\lambda|<\epsilon\}$. Again, $U$ being an open set containing $x$, $\exists$ a balanced neighbourhood $U_x $ of $\theta$ such that $x \in x+U_x \subseteq U$ [by result \ref{ch4:r:3}].  Hence, $B(\alpha,\epsilon).(x+U_x)$ $\subseteq$ $\alpha G.$
\end{proof}

\begin{Th} In a topological evs $X$ over the field $\mathbb{K}$ of real or complex numbers, there exists a local base $\mathscr{U}$ at $\theta$ such that the following conditions hold:

(i) $U \in \mathscr{U} \Longrightarrow U$ is balanced and absorbing.

(ii) $U \in \mathscr{U} \Longrightarrow \exists\, V \in \mathscr{U}$ such that $V + V \subseteq U.$

(iii) $U_1, U_2 \in \mathscr{U} \Longrightarrow \exists\, W \in \mathscr{U} \text{ such that } W \subseteq U_1 \cap U_2.$

(iv) For any $x, y \in X$ with $x \nleq y$, $\exists$ $U, V \in \mathscr{U}$ such that  $\uparrow(x + U)$ $\cap$ $\downarrow(y + V)$ = $\emptyset$.

(v) For any open set $G$ containing $x$ and $\alpha \in \K\smallsetminus\{0\}$, $\exists\, U_x \in \mathscr{U}$ and $\epsilon>0$ such that $B(\alpha,\epsilon).(x+U_x)$ $\subseteq$ $\alpha G$ where, $B(\alpha,\epsilon):=\{\lambda\in\K:|\alpha-\lambda|<\epsilon\}$.
\end{Th}

\begin{proof} Let $\mathscr{U}$ be the set of all balanced neighbourhoods of $\theta$. Then, (i), (ii), (iv) and (v) follows from results \ref{ch4:r:1}, \ref{ch4:r:2}, \ref{ch4:r:4} and \ref{ch4:r:5} respectively. Also, (iii) follows as $U_1, U_2 \in \mathscr{U} \Longrightarrow U_1 \cap U_2\in \mathscr{U}$ [ in view of proposition \ref{ch4:p:b1} (ii) ].
\end{proof}

The next theorem gives a sufficient condition for an evs to be topological.

\begin{Th} Let X be an evs over the field $\mathbb{K}$ of real or complex numbers such that $\exists$ a non-void family $\mathscr{U}$ of subsets of $X$ satisfying the following conditions:

(i) $U \in \mathscr{U}$ $\Longrightarrow$ $U$ is balanced and absorbing.

(ii) $U \in \mathscr{U}$ $\Longrightarrow$ $\exists\, V \in \mathscr{U}$ such that $V + V \subseteq U.$

(iii) $U_1, U_2 \in \mathscr{U}$ $\Longrightarrow$ $\exists$ $W \in \mathscr{U}$ such that $W \subseteq U_1 \cap U_2.$

(iv) For any $x, y \in X$ with $x \nleq y$, $\exists$ $U, V \in \mathscr{U}$ such that  $\uparrow(x + U)$ $\cap$ $\downarrow(y + V)$ = $\emptyset$.

(v) For any $W \in \mathscr{U}$, $x \in X$ and $\alpha \in \K$, $\exists$ $\epsilon$ > 0 and $U \in \mathscr{U}$ such that $B(\alpha, \epsilon)\cdot (x + U)$ $\subseteq$ $\alpha x + W$ where, $B(\alpha, \epsilon):=\{\lambda\in\K:|\alpha-\lambda|<\epsilon\}$ is an open disc in $\mathbb{K}$ having center at $\alpha$ and radius $\epsilon.$

Then, $\exists$ a unique topology on $X$ depending upon $\mathscr{U}$ with respect to which $X$ becomes a topological evs. Moreover, $\mathscr{U}$ forms a local base at $\theta$ in that topology.

\label{ch4:t:sc}\end{Th}

\begin{proof} Let $X$ be an evs over the field $\mathbb{K}$ in which the conditions (i) - (v) stated above are satisfied by a non-void family $\mathscr{U}$ of subsets of $X$. For each $x \in X,$ let us define $\mathscr{U}_x := \{x + U : U \in \mathscr{U}\}.$ 

Since $\mathscr{U} \neq \emptyset$, so $\mathscr{U}_x \neq \emptyset$ for any $x \in X.$ 

Since each $U \in \mathscr{U}$ is absorbing, so $\theta \in U,  \forall\ U \in \mathscr{U}$ $\Longrightarrow$ $x \in x + U$, $\forall( x + U)\in \mathscr{U}_x.$ 

Let $V_1, V_2 \in \mathscr{U}_x$. So, $V_1$ = $x+U_1$, $V_2$ = $x+U_2,$ where $U_1, U_2 \in \mathscr{U}$. Now, by (iii), $\exists\, V \in \mathscr{U}$ such that $V \subseteq U_1 \cap U_2.$ So, $x+V \subseteq x+(U_1 \cap U_2)\subseteq V_1 \cap V_2 $ where, $x+V \in \mathscr{U}_x.$

Let $V_x$ = $x+V \in \mathscr{U}_x$. Since $V \in \mathscr{U}$ so by (ii), $ \exists V_1 \in \mathscr{U}$ such that $V_1+V_1 \subseteq V.$ Clearly, $V_1 \subseteq V$ as $\theta \in V_1$. Therefore $x+V_1 \subseteq x+V.$ Put $W_x$ = $x+V_1.$ Then $W_x \in \mathscr{U}_x$ so that $x \in W_x \subseteq V_x.$ Now, for any $y \in W_x, V_y$ = $y+V_1 \in \mathscr{U}_y$ such that for $z \in V_y, z$ = $y+t$ for some $t \in V_1.$ But, $y \in W_x$ $\Longrightarrow$ $y$ = $x+t_1$ for some $t_1 \in V_1.$ Thus, $z$ = $x+(t_1+t)$ $\in x+V_1+V_1$ $\subseteq$ $x+V=V_x$. Therefore $y \in V_y$  $\subseteq V_x$ where, $y \in W_x$ is arbitrary and $V_y \in \mathscr{U}_y.$ 

From above discussion it follows that $\exists$ a unique topology $\tau$ on $X$ with respect to which $\mathscr{U}_x$ is a local base at $x,  \forall x \in X.$ Taking $x=\theta$ we get, $\mathscr{U}$ is a local base at $\theta.$

We now prove that $(X, \tau)$ is a topological evs over $\K$.

Let $f_1:X\times X \rightarrow X$ be defined by: $f_1(x, y)=x+y$, $\forall x,y\in X.$ Let $x, y \in X$ and $G$ be any neighbourhood of $x+y$. Since $\mathscr{U}_{x+y}$ is a local base at $x+y, \exists$ $V \in \mathscr{U}$ such that $(x+y)+V \subseteq G.$ By (ii), $\exists$ $V_1 \in \mathscr{U}$ such that $V_1+V_1$ $\subseteq V$. Therefore $x+y+V_1+V_1$ $\subseteq x+y+V \subseteq G$ $\Longrightarrow$ $(x+V_1)+(y+V_1)$ $\subseteq$ $G$ where, $x+V_1 \in \mathscr{U}_x$, $y+V_1 \in \mathscr{U}_y$ i.e. $f_1((x+V_1)\times(y+V_1))$ $\subseteq G.$ Since $x, y \in X$ are arbitrary, $f_1$ is continuous on $X\times X.$

Let $f_2:\mathbb{K}\times X\rightarrow X$ be defined by: $f_2(\alpha,x)$ = $\alpha.x$, $\forall \alpha \in \mathbb{K}$, $\forall x \in X$. Let $\lambda \in \mathbb{K}$, $y \in X$ be arbitrary and $G$ be an open set containing $\lambda y.$ Then, $\exists$ $W \in \mathscr{U}$ such that $\lambda y+W \subseteq G$ [ $\because$ $\mathscr{U}_{\lambda y}$ is a local base at $\lambda y$ ]. Now, $W \in \mathscr{U}$, $\lambda \in \mathbb{K}$ and $y \in X$ so by (v), $\exists$ $\epsilon$ > 0 and $U \in \mathscr{U}$ such that $f_2(B(\lambda,\epsilon)\times(y+U))$ $\subseteq$ $\lambda y+W \subseteq G$ where, $B(\lambda,\epsilon)$ is a neighbourhood of $\lambda$ and $y+U$ is a neighbourhood of $y$. Hence $f_2$ is continuous.

Let $x, y \in X$ with $x \nleq y$. Then by (iv), $\exists U_x, U_y \in \mathscr{U}$ such that $\uparrow(x + U_x)$ $\cap$ $\downarrow(y + U_y)$ = $\emptyset$. Now, $x+U_x, y+U_y$ are neighbourhoods of $x$ and $y$ respectively. So, $\exists$ open sets $G_x$ and $G_y$ containing $x$ and $y$ respectively such that $x \in G_x \subseteq x+U_x$ and $y \in G_y \subseteq y+U_y.$ Clearly, $\uparrow G_x$ $\cap$ $\downarrow G_y$ = $\emptyset$. So, the partial order `$\leq$' is closed [By theorem \ref{ch1:t:closedorder}].

Hence $(X,\tau)$ is a topological evs over the field $\K$.
\end{proof}

\begin{Lem} Let $X$ and $Y$ be order-isomorphic evs over the field $\K$ with $\phi:X \rightarrow Y$ as an order-isomorphism between them. Then, $\mathscr{U}_Y=\{\phi(U):U \in\mathscr{U}_X\}$ where, $\mathscr{U}_X$ and $\mathscr{U}_Y$ denote the collection of all absorbing sets in the evs $X$ and $Y$ respectively.
\label{ch4:l:1}\end{Lem}

\begin{proof}Let $W \in \mathscr{U}_X$ and $y \in Y$ be arbitrary. Then $y=\phi(x)$ for some $x\in X.$ Now, $W$ is absorbing so $\exists$ $\epsilon$ > 0 such that $\mu x \in W$, $\forall \mu \in \mathbb{K}$ with $| \mu |$ $\leq$ $\epsilon$ $\Longrightarrow$ $\phi(\mu x) \in \phi(W)$, $\forall \mu \in \mathbb{K}$ with $| \mu |$ $\leq$ $\epsilon\Rightarrow\mu \phi(x) \in \phi(W)$, $\forall \mu \in \mathbb{K}$ with $| \mu |$ $\leq$ $\epsilon\Rightarrow\mu y \in \phi(W)$, $\forall \mu \in \mathbb{K}$ with $| \mu |$ $\leq$ $\epsilon$. $\therefore$ $\phi(W)$ is an absorbing set $\Longrightarrow$ $\phi(W)\in \mathscr{U}_Y.$

Now, let $V \in \mathscr{U}_Y$, $x \in X$ be arbitrary and $U$ = $\phi^{-1}(V)$. Since $V$ is absorbing, so $\exists$ $\epsilon$ > 0 such that $\mu \phi(x) \in V$, $\forall \mu \in \mathbb{K}$ with $| \mu |$ $\leq$ $\epsilon \Longrightarrow$ $\phi(\mu x) \in V$, $\forall \mu \in \mathbb{K}$ with $| \mu |$ $\leq$ $\epsilon \Longrightarrow$ $\mu x \in \phi^{-1}(V)$ = $U$, $\forall \mu \in \mathbb{K}$ with $| \mu |$ $\leq$ $\epsilon$. Since $x \in X$ is arbitrary so $U \in \mathscr{U}_X$ $\Longrightarrow$ $V \in \{\phi(U):U \in\mathscr{U}_X\}.$ Hence $\mathscr{U}_Y$ = $\{\phi(U):U \in\mathscr{U}_X\}$.
 \end{proof}

\begin{Th} Let $X$ and $Y$ be two order-isomorphic evs over the field $\K$. Then, $X$ has a non-void family $\mathscr{U}_X$ of subsets of $X$ satisfying conditions (i)-(v) of theorem \ref{ch4:t:sc} iff $Y$ has a non-void family $\mathscr{U}_Y$ of subsets of $Y$ satisfying the same. Hence, the property of an evs to have a non-void family of subsets satisfying above mentioned conditions, is an evs property. 
\end{Th}

\begin{proof} Let $\phi : X \rightarrow Y$ be an order-isomorphism and $\mathscr{U}_X$ be a family of non-void subsets of $X$ satisfying conditions (i)-(v) of theorem \ref{ch4:t:sc}. Let, $\mathscr{U}_Y := \{\phi(U):U \in \mathscr{U}_X\}.$ Since $\mathscr{U}_X \neq \emptyset$, so $\mathscr{U}_Y \neq \emptyset$.

Let $W \in \mathscr{U}_X$. So, $W$ is an absorbing and balanced set. Therefore, by lemma \ref{ch4:l:1}, $\phi(W)$ is an absorbing subset of $Y$. For any $\lambda \in \mathbb{K}$ with $|\lambda|\leq1$, $\lambda W$ $\subseteq$ $W$ [ $\because$ $W$ is a balanced set ] $\Longrightarrow$ $\phi(\lambda W)$ $\subseteq$ $\phi(W)\Longrightarrow\lambda\phi(W)$ $\subseteq$ $\phi(W)$. Therefore, $\phi(W)$ is a balanced subset of $Y$.

So, $\mathscr{U}_Y$ satisfies condition (i).

Again, $W \in \mathscr{U}_X$ $\Longrightarrow$ $\exists$ $V \in \mathscr{U}_X$ such that $V+V$ $\subseteq$ $W$. Therefore $\phi(V+V)$ $\subseteq$ $\phi(W)$  $\Rightarrow\phi(V)+\phi(V)$ $\subseteq$ $\phi(W)$ $\big[ \because \phi(V+V)$ = $\phi(V)+\phi(V) \big]$ where, $\phi(V) \in \mathscr{U}_Y.$ So, $\mathscr{U}_Y$ satisfies condition (ii).

Let $\phi(U_1)$, $\phi(U_2) \in\mathscr{U}_Y$ for some $U_1, U_2 \in \mathscr{U}_X$. Therefore  $\exists$ $W \in \mathscr{U}_X$ such that $W \subseteq (U_1 \cap U_2)$ $\Longrightarrow$ $\phi(W)$ $\subseteq$ $\phi(U_1)$ $\cap$ $\phi(U_2).$ So, $\mathscr{U}_Y$ satisfies condition (iii).

Since, $\phi$ is an order-isomorphism, so $x \leq y$ in $X$ iff $\phi(x)$ $\leq$ $\phi(y)$ in $Y$. For any $x \in X$, $U \in \mathscr{U}_X,$ we have, $z \in \downarrow$ $(x+U)$ iff $z \leq x+u$ for some $u \in U$ iff $\phi(z)$ $\leq$ $\phi(x+u)$ = $\phi(x)+\phi(u)$ iff $\phi(z) \in$ $\downarrow(\phi(x)+\phi(U))$. Thus $ \phi(\downarrow(x+U))$ = $\downarrow(\phi(x)+\phi(U)).$ Similarly, $\phi(\uparrow(x+U))$ = $\uparrow(\phi(x)+\phi(U)).$ Let $\phi(x_1)$ = $y_1$ and $\phi(x_2)$ = $y_2 \in Y$ with $y_1 \nleq y_2$. Then $ x_1 \nleq x_2$ $\Longrightarrow$ $\exists$ $U_1, U_2 \in \mathscr{U}_X$ such that $\uparrow (x_1 + U_1)$ $\cap$ $\downarrow (x_2 +U_2)=\emptyset$ $\Longrightarrow$ $\phi(\uparrow (x_1 + U_1)\cap\downarrow (x_2 +U_2))=\emptyset$ $\Longrightarrow$ $\phi (\uparrow (x_1 + U_1) )$ $\cap$ $\phi(\downarrow (x_2 +U_2))=\emptyset$ $\Longrightarrow$ $\uparrow (\phi(x_1) + \phi(U_1) )\cap\downarrow (\phi(x_2) +\phi(U_2))$ = $\emptyset$ i.e. $\uparrow (y_1 + \phi(U_1) )$ $\cap$ $\downarrow (y_2 +\phi(U_2))$ = $\emptyset$ where $\phi(U_1)$ and $\phi(U_2) \in \mathscr{U}_Y.$ Hence, $\mathscr{U}_Y$ satisfies condition (iv).

Let $ \phi(W) \in \mathscr{U}_Y$ for some $W \in \mathscr{U}_X, y$ = $\phi(x) \in Y$ for some $x\in X$ and $\alpha \in \mathbb{K}.$ Since, $\mathscr{U}_X$ satisfies condition (v) of theorem \ref{ch4:t:sc} so $\exists$ $\epsilon$ > 0 and $U_x \in \mathscr{U}_X$ such that $B(\alpha, \epsilon)\cdot (x + U_x)$ $\subseteq$ $\alpha x + W.$ Now, $\phi$ is an order-isomorphism so, $\phi(B(\alpha, \epsilon)\cdot (x + U_x))$ = $B(\alpha, \epsilon)\cdot \phi(x + U_x)$ = $B(\alpha, \epsilon)\cdot (\phi(x) + \phi(U_x))$ = $B(\alpha, \epsilon)\cdot(y + \phi(U_x))$. Again, $\phi(B(\alpha, \epsilon)\cdot (x + U_x))$ $\subseteq$ $\phi(\alpha x +W)$ = $\phi(\alpha x)+\phi(W)$ = $\alpha\phi(x)+\phi(W)$ = $\alpha y +\phi(W).$

Thus, for $\phi(W) \in \mathscr{U}_Y, y \in Y$ and $\alpha \in \mathbb{K}$, $\exists$ $\epsilon>0$ and $\phi(U_x) \in \mathscr{U}_Y$ such that\\ $B(\alpha, \epsilon)\cdot (y + \phi(U_x))$ $\subseteq$ $\alpha y +\phi(W)$. Hence $\mathscr{U}_Y$ satisfies condition (v).

 Similarly, if $\mathscr{U}_Y$ be a family of non-void subsets of $Y$ satisfying conditions (i)-(v) of theorem \ref{ch4:t:sc} then taking $\mathscr{U}_X := \{\phi^{-1}(U):U \in \mathscr{U}_Y\}$ we can prove that $\mathscr{U}_X$ satisfies conditions (i)-(v) of theorem \ref{ch4:t:sc} in $X$.
\end{proof}

\section{Bounded Sets in a Topological evs }\markright{Bounded Sets in a Topological evs}

In this section we have introduced the concept of bounded sets in a topological evs over the field $\K$ and characterised them with the help of balanced sets.

\begin{Def} Let $X$ be a topological evs over the field $\K$ of real or complex numbers. A set $A(\subseteq X$) is said to be \textit{bounded}\index{Bounded set} if for any neighbourhood $V$ of $\theta$, $\exists$ a real number $\alpha$ > 0 such that $\mu A \subseteq V$, $\forall \mu \in \mathbb{K}$ with $|\mu|$ $\leq$ $\alpha.$
\label{ch4:d:bdd}\end{Def}  

\begin{Res} A set $B$ in a topological evs $X$ over the field $\K$ is bounded iff for any balanced neighbourhood $U$ of $\theta$, $\exists$ $\alpha$ > 0 such that $B\subseteq\alpha U.$
\label{ch4:r:8}\end{Res}

\begin{proof} Let $B$ be a bounded set in $X$ and $U$ be a balanced neighbourhood of $\theta$. So, $\exists$ $\beta$ > 0 such that $\mu B \subseteq U$, $\forall \mu \in \mathbb{K}$ with $|\mu|$ $\leq$ $\beta$. In particular, $\beta B$ $\subseteq$ $U$ $ \Rightarrow $ $B\subseteq \frac{1}{\beta}U$ $ \Rightarrow $ $B\subseteq\alpha U$ where, $\alpha = \frac{1}{\beta}>0.$ Conversely, let the given condition hold and $U$ be any neighbourhood of $\theta$. Then by result \ref{ch4:r:1}, $\exists$ a balanced neighbourhood $V$ of $\theta$ such that $V\subseteq U$. Also, $\exists$ $\alpha>0$ such that $B\subseteq\alpha V$. Let $\lambda=\frac{1}{\alpha}>0$. Then for $\mu\in\K$ with $|\mu|\leq\lambda$ we have, $|\mu\alpha|\leq1$ . So $\mu\alpha V \subseteq V\ (\because V$ is a balanced set). Thus for $|\mu|\leq\lambda$ we have, $\mu B\subseteq \mu\alpha V\subseteq V \subseteq U.$ Hence $B$ is a bounded set.
\end{proof}

\begin{Not}
	In view of result \ref{ch4:r:8} we can state:
	 `In a topological evs $X$ over the field $\K$, a subset of a bounded set is bounded'.
	\label{ch4:n:8}\end{Not}

\begin{Res} A subset $A$ of a topological evs $X$ over the field $\K$ is bounded iff for any sequence $\{x_n\}$  in $A$ and a sequence $\{\lambda_n\}$ in $\mathbb{K}$ with $\lambda_n\rightarrow 0$ we have, $\lambda_nx_n \rightarrow \theta.$
\label{ch4:r:nscb}\end{Res}

\begin{proof} Let $A$ be a bounded subset of $X$ and $U$ be any neighbourhood of $\theta$. Then $\exists$ $\alpha>0$ such that $\mu A\subseteq U$, $\forall \mu\in\K$ with $|\mu|$ $\leq$ $\alpha$ $\cdots\cdots$ (1). Let $\{x_n\}$ be a sequence in $A$ and $\{\lambda_n\}$ be  a sequence in $\mathbb{K}$ with $\lambda_n\rightarrow 0$.
Since $\lambda_n\rightarrow0$, so for $\alpha > 0$, $\exists$ $k\in \N$ such that $|\lambda_n|<\alpha$, $\forall n \geq k$ $\Longrightarrow$ $\lambda_n x_n\in \lambda_n A\subseteq U$, $\forall n \geq k$ [ by (1) ]. As $U$ is arbitrary neighbourhood of $\theta$, so $\lambda_n x_n\rightarrow \theta.$

Conversely, let the given condition hold. If possible, let $A$ be not bounded. Then $\exists$ a neighbourhood $U$ of $\theta$ such that for $\frac{1}{n}>0$ ( where, $n\in\N$ ), $\exists \lambda_n\in\K\smallsetminus\{0\}$ such that $|\lambda_n|$ $\leq$ $\frac{1}{n}$ and $\lambda_n A$ $\nsubseteq$ $U\Rightarrow\exists$ $x_n\in A$ such that $\lambda_nx_n\notin U,\forall\,n\in\N$. Thus $\{x_n\}$ is a sequence in $A$ such that $\lambda_n x_n$ $\nrightarrow\theta$ where $\lambda_n\rightarrow0$, which is a contradiction. Hence $A$ is a bounded set.
\end{proof}

\begin{Res} Every finite subset of a topological evs $X$ over the field $\K$ is bounded.
\label{ch4:r:6}\end{Res}

\begin{proof} Let $A= \{x_1,x_2,...,x_n\}$ be a finite set in $X$ and $V$ be a neighbourhood of $\theta$. By result \ref{ch4:r:1}, $V$ is absorbing. So, for each $x_i\in A$ $\big(\text{where } i\in\{1,2,...,n\}\big)$, $\exists$ a real number $\alpha_i>0$ such that $\mu x_i \in V$, $\forall \mu\in \mathbb{K}$ with $|\mu|$ $\leq$ $\alpha_i$. Let $\alpha=\min\{\alpha_i:i=1,2,...,n\}$. Then, $\alpha>0$ and $\mu A\subseteq V$, $\forall \mu\in \mathbb{K}$ with $|\mu|\leq\alpha $. Hence $A$ is bounded.
\end{proof}

\begin{Res} Every compact subset in a topological evs $X$ over the field $\K$ is bounded.
	\label{ch4:r:compact}\end{Res}

\begin{proof} 
	Let $X$ be a topological evs over the field $\K$ and $A$ be a compact subset of $X$. Let $V$ be a balanced open neighbourhood of $\theta$. Then, for each $x\in A, \exists$ a real number $\alpha_x>0$ such that $\mu x\in V$, $\forall \mu\in \mathbb{K}$ with $|\mu|\leq\alpha_x$ [ $\because$ $V$ is absorbing ]. Let $n_x\in\N$ be such that $0<\frac{1}{n_x}<\alpha_x$, $\forall x\in A$. Then, $\frac{1}{n_x}x\in V$, $\forall x\in A$ $\Longrightarrow$ $x\in n_x V$, $\forall x\in A$. Hence $\{ n_xV: x\in A \}$ is an open cover of the compact set $A$. Let $\{ n_xV: x\in A' \}$ [ where, $A'$ is a finite subset of $A$ ] be a finite subcover of $A$. So, $A\subseteq \underset{x\in A'}{\bigcup}n_xV$. Now, let $n:=\max \{n_x:x\in A' \}$. Then, $n_x\leq n$, $\forall x\in A'$ $\Longrightarrow$ $\frac{n_x}{n}\leq1$, $\forall x\in A'$ $\Longrightarrow$ $\frac{n_x}{n}V\subseteq V$, $\forall x\in A'$ [ $\because$ $V$ is balanced ] $\Longrightarrow$ $n_xV\subseteq nV$, $\forall x\in A'$ $\Longrightarrow$ $\underset{x\in A'}{\bigcup}n_xV\subseteq nV$. Thus, $A\subseteq nV$ where, $n>0$. Hence $A$ is a bounded set [ in view of result \ref{ch4:r:8} ].
\end{proof}

\begin{Res} If $A$, $B$ are bounded subsets of a topological evs $X$ over the field $\K$ then $A+B, \lambda A$ are also bounded sets for any $\lambda \in \mathbb{K}.$
\label{ch4:r:7}\end{Res}

\begin{proof} Let $X$ be a topological evs over the field $\K$ and $A$, $B$ be two bounded subsets of $X$. Let $V$ be any neighbourhood of $\theta$. Then by result \ref{ch4:r:2}, $\exists$ a neighbourhood $W$ of $\theta$ such that $W+W\subseteq V$. For $W$, $\exists$ $\alpha, \beta>0$ such that $\mu A\subseteq W$, $\forall \mu\in \mathbb{K}$ with $|\mu|$ $\leq$ $\alpha$ and $\mu B\subseteq W$, $\forall \mu\in \mathbb{K}$ with $|\mu|$ $\leq$ $\beta$. Therefore $\mu(A+B)$ = $\mu A+\mu B$ $\subseteq$ $W+W$ $\subseteq V$, $\forall \mu\in \mathbb{K}$ with $|\mu|$ $\leq$ $\lambda$, where $\lambda=\min \{\alpha, \beta\}$. So, $A+B$ is a bounded set.  

Clearly, if $\lambda =0$ then $\lambda A=\{\theta\}$ which is a bounded set. Let $\lambda \in \mathbb{K}\smallsetminus\{0\}$ and $U$ be any neighbourhood of $\theta$. Then $\exists$ $\alpha>0$ such that $\mu A\subseteq U,$ for $|\mu|\leq\alpha$. Put, $\beta$ = $\frac{\alpha}{|\lambda|}>0.$ Then $|\mu|\leq\beta$ $\Longrightarrow$ $|\mu\lambda|$ =  $|\mu||$ $\lambda$ $|$ $\leq$ $\beta$ $|\lambda|$ = $\alpha$. So $\mu(\lambda A)$ = $(\mu\lambda)A\subseteq U$ for $|\mu|\leq\beta.$ Hence $\lambda A$ is a bounded set.
\end{proof}

\begin{Not} Let $X$ be a topological evs over the field $\K$. For any $x\in X_0$, the primitive $P_x$ of $x$ is the singleton set $\{x\}$. So, $P_x$ is a bounded set for any $x\in X_0.$
\end{Not}

\section{Some Applications of Absorbing and Balanced Sets}\markright{Some Applications of Absorbing and Balanced Sets}

 Consider the evs $[0,\infty)$ over the field $\K$ (example \ref{ch1:e:0inf}). For any $r>0$, $[0, r)$ is an absorbing set in  $[0,\infty)$ $\big[\because$ for any $x>0$, $\exists\ \epsilon>0$ such that $\epsilon x < r$ and so for any $\alpha \in\K$ with $|\alpha|\leq\epsilon$ we have $|\alpha|x \leq \epsilon x<r$ i.e. $\alpha x \in [0, r)$ $\forall \alpha\in\K$ with $|\alpha|\leq\epsilon$  $\big]$. Let $x,y \in [0,\infty)$ with $x\neq y.$ Without loss of generality, let $x<y$. So, $\exists\,r>0$ such that $x<r<y$ . Therefore  $[0,r)$ is an absorbing set containing $x$ but not containing $y.$ Thus, in the evs $[0,\infty),$ for any two distinct points, $\exists$ an absorbing set containing one point but not the other.
 
 Let $\X$ be a vector space over the field $\K$ of real or complex numbers. Let
 $\mathscr L(\X)$ be the evs of all linear subspaces of $\X$ over $\K$ (example \ref{ch1:e:LX}). Let $\mathscr U$ be any proper subset of $\mathscr {L}(\X)$ containing $\{\theta\}$ and $Y\in \mathscr {L}(\X)\smallsetminus\mathscr U.$ Then for any $\alpha>0$, $\alpha Y=Y \notin\mathscr U \Rightarrow\mathscr U$ cannot be an absorbing set. Hence the only absorbing set in $\mathscr{L}(\X)$ is $\mathscr {L}(\X)$ itself. So for two distinct points of $\mathscr {L}(\X)$, there does not exist an absorbing set that contains one of the points but not the other.
 
 Thus, in some evs, the set of all absorbing sets can distinguish points in the above sense and in some evs it does not do so. Therefore, we can make the following definition:
 
 \begin{Def} An evs $X$ over the field $\mathbb{K}$ is said to be \emph{radial}\index{Radial evs} if for any $x, y \in X$ with $x \neq y, \exists$ an absorbing set which contains any one of $x$ and $y$ but not the both. 
 \end{Def}
 
 \begin{Th} The property of an evs to be radial is an evs property.
 \end{Th}
 \begin{proof} Let $X$ and $Y$ be order-isomorphic evs over the same field $\K$ and $\phi:X\rightarrow Y$ be an order-isomorphism. Let $X$ be a radial evs and $y_1 = \phi(x_1), y_2 =\phi(x_2) \in Y$ for some $x_1, x_2 \in X$ with $y_1\neq y_2.$ Then $x_1\neq x_2$. Since $X$ is radial, $\exists$ an absorbing set $A$ in $X$ such that $A$ contains one of $x_1$ and $x_2$ but not the both. Without loss of generality, let, $x_1\in A$ but $x_2 \notin A$ . Then $y_1 \in \phi(A)$ but $y_2 \notin \phi(A).$ Now, by lemma \ref{ch4:l:1}, $\phi(A)$ is an absorbing set. Hence $Y$ is also a radial evs. Thus, the property of an evs to be radial is an evs property.
 \end{proof}
 
 \begin{Th} The property of an evs to be radial is a productive property.
 \end{Th}
 
 \begin{proof} Let $\{X_\alpha : \alpha \in \Lambda\}$ be a collection of radial evs over the same field $\mathbb{K}$ where, $\Lambda$ is an index set and $X = \underset{\alpha \in \Lambda}{\Pi}X_\alpha$. Let $x=(x_\alpha)_{\alpha \in \Lambda}$, $y=(y_\alpha)_{\alpha \in \Lambda} \in X$ with $x \neq y$. Then $\exists\,\lambda \in \Lambda$ such that $x_\lambda \neq y_\lambda$. Since $X_\lambda$ is radial, $\exists$ an absorbing set $A_\lambda$ in $X_\lambda$ such that $A_\lambda$ contains any one of $x_\lambda$ and $y_\lambda$ but not the both. Let $A_\lambda$ contains $x_\lambda$ but not $y_\lambda$. Then  $x \in \underset{\alpha \in \Lambda}{\Pi}Y_\alpha$ but $y \notin \underset{\alpha \in \Lambda}{\Pi}Y_\alpha$ where, $Y_\alpha = X_\alpha$ for $\alpha \in \Lambda \smallsetminus\{\lambda\}$ and $Y_\lambda = A_\lambda.$ Now, it is enough to prove that $\underset{\alpha \in \Lambda}{\Pi}Y_\alpha$ is an absorbing set in $X$. Let $z=(z_\alpha)_{\alpha \in \Lambda} \in X$. Now, $z_\lambda \in X_\lambda$ and $A_\lambda$ is an absorbing set. So, $\exists \beta>0$ such that $\mu z_\lambda \in A_\lambda, \forall \mu \in \mathbb{K}$ with $|\mu|\leq\beta$. Also, for any $\alpha \in \Lambda\smallsetminus \{\lambda\}$, $\mu z_\alpha \in X_\alpha$, $\forall \mu \in \mathbb{K}$ with $|\mu|\leq\beta$. Thus, $\mu z= (\mu z_\alpha)_{\alpha \in \Lambda} \in  \underset{\alpha \in \Lambda}{\Pi}Y_\alpha,  \forall \mu \in \mathbb{K}$ with $|\mu|\leq\beta$ . Therefore  $\underset{\alpha \in \Lambda}{\Pi}Y_\alpha$ is an absorbing set. So, $X$ is also a radial evs. Hence the property of an evs to be radial is a productive property.
 \end{proof}
 
 \begin{Th} The property of an evs to be radial is a hereditary property.
 \end{Th}
 \begin{proof} Let $X$ be a radial evs over the field $\mathbb{K}$ and $Y$ be any subevs of $X$. Let $x, y \in Y$ with $x \neq y$. Since $X$ is radial, $\exists$ an absorbing set $A\subset X$ such that $A$ contains any one of $x$ and $y$ but not the both. Let $x\in A$ but $y \notin A$. Put, $A_Y = A\cap Y$. Clearly, $\theta \in A_Y$ i.e $A_Y\neq \emptyset$. Also, $x\in A_Y, y\notin A_Y.$ Let, $z \in Y$ be arbitrary. Since, $A$ is absorbing, $\exists \beta>0$ such that  $\mu z \in A, \forall \mu \in \mathbb{K}$ with $|\mu|\leq\beta$. Since $Y$ is a subevs so $\mu z \in Y$ for any $\mu\in\mathbb{K}$ . Therefore  $\mu z \in A_Y,  \forall \mu \in \mathbb{K}$ with $|\mu|\leq\beta$. Thus, $A_Y$ is an absorbing set in the evs $Y \Rightarrow Y$ is also radial.  Hence the property of an evs to be radial is a hereditary property.
 \end{proof}
 
 \begin{Res} Every topological evs $X$ over the field $\mathbb{K}$ is radial.
 \end{Res}

 \begin{proof} Let $X$ be a topological evs over the field $\mathbb{K}$. Then, by result \ref{ch4:r:1}, any neighbourhood of $\theta$ is absorbing. Let $x, y \in X$ with $x\neq y$. Without loss of generality, let $y\neq \theta.$  Since $X$ is a topological evs, it is $T_2$. So, $\exists$ a neighbourhood $U$ of $\theta$ which does not contain $y$. Now, $U$ is absorbing $\Rightarrow U\cup \{x\}$ is also absorbing (By propoition \ref{ch4:p:a1} (iii)). Thus, $U\cup \{x\}$ is an absorbing set containing $x$ but not containing $y$ where, $x,y\in X$ are arbitrary.  Therefore, $X$ is a radial evs.
 \end{proof}
 
 \begin{Not} Converse of the above result is not true. Consider the non-topological evs $X = \mathscr{D}^2[0,\infty)$ (example \ref{ch1:e:d2inf}). Consider the set $[0,r_1)\times[0,r_2) \in X$ where $r_1, r_2>0$. Let $(x, y)\in X$ be arbitrary. We can choose $\epsilon>0$ such that $\epsilon x<r_1$ and $\epsilon y<r_2$. i.e. $\epsilon(x, y)\in[0,r_1)\times[0,r_2)$. Also, for any $\mu\in \mathbb{K}$ with $|\mu|\leq\epsilon$,  $\mu(x,y)=(|\mu| x, |\mu| y)\in [0,r_1)\times[0,r_2)$ . Therefore  $[0,r_1)\times[0,r_2)$ is an absorbing set. Now, let $(x_1,y_1), (x_2, y_2)\in X$ with $(x_1,y_1)\neq(x_2,y_2)$. As, $x_1, x_2, y_1, y_2$ cannot be all zero, without loss of generality, let $x_1\neq x_2$ and $x_1>x_2.$ Let us choose $r>0$ such that $x_2<r<x_1$. Then, $[0,r)\times[0,y_2+1)$ is an absorbing set containing $(x_2,y_2)$ but not containing $(x_1,y_1).$ Since $(x_1,y_1), (x_2, y_2)\in X$ are arbitrary, so $X$ is a radial evs.
 \end{Not}
  
 We now find the exact form of an open, balanced and absorbing set in the topological evs $[0,\infty)$.

\begin{Lem} In the evs $[0,\infty)$ over the field $\K$, a set is balanced and absorbing iff it is an interval containing  0.
\label{ch4:l:3}\end{Lem}

\begin{proof}Let $U$ be a non-empty balanced and absorbing set in the evs $[0,\infty)$ over the field $\mathbb{K}$. By proposition \ref{ch4:p:b1}, $0\in U.$ Since $U$ is absorbing, $\exists\ x \in U$ with $x>0$ $(\because \{0\}$ is not an absorbing set). Let $0<y<x$. Then $y=\frac{y}{x}.x \in U$ ($\because\ |\frac{y}{x}|<1$ and $U$ is balanced). So, $[0,x]\subseteq U$ for any $x \in U$. Hence $U$ is an interval containing 0. 

Conversely, let $U$ be an interval containing 0. Let $z\in [0,\infty)$ and $y\in U$ with $y>0$. We can find $\epsilon >0$ such that $\epsilon z <y$. So, for any $\mu\in\mathbb{K}$ with $|\mu|\leq\epsilon$ we have, $|\mu|z \leq \epsilon z<y$ i.e. $\mu z \in U$ for $|\mu|\leq\epsilon \Rightarrow U$ is absorbing. Also, for any $\alpha\in\mathbb{K}$ with $|\alpha|\leq 1$ and any $x\in U, |\alpha|x\leq x\Rightarrow\alpha x\in U \Rightarrow\alpha U\subseteq U\Rightarrow U$ is balanced. Hence in $[0,\infty)$, a set is balanced and absorbing iff it is an interval containing  0.
\end{proof}

\begin{Lem} If $[0,\infty)$ is a topological evs with a topology (may be other than the usual subspace topology) defined on it then an open, balanced and absorbing set $`U$' is of the form : $U=[0,a)$ where, $a$ may be a finite number or $\infty$.
\label{ch4:l:4}\end{Lem}

\begin{proof} Let $U$ be an open, balanced and absorbing set in the topological evs $[0,\infty)$ with some topology defined on it. By lemma \ref{ch4:l:3}, $U$ is an interval containing 0. If possible, let $U=[0,a]$ for some $a\in (0,\infty)$. Since, $U$ is open and $1.a=a\in U$ so, $\exists$ an open disc $B(1,\epsilon)$ in $\mathbb{K}$ such that $B(1,\epsilon).a\subseteq U$ \big[since the scalar multiplication is continuous as, $ [0,\infty) $ is a topological evs\big]. Choose $t\in B(1,\epsilon)$ with $| t | >1.$ Then $t.a=| t |a >a \in U$ ------ which is a contradiction, since $ U=[0,a] $. Hence, $U$ must be of the form $[0,a)$ where, $a$ may be a finite number or $\infty$.
\end{proof}

\begin{Th} The usual subspace topology is the finest topology with respect to which $[0,\infty)$ forms a topological evs over the field $\K$. 
\label{ch4:t:[0,infty)}\end{Th}
	\begin{proof} Let $\tau$ denote the usual subspace topology on $[0,\infty)$ and $\tau^ \prime$ be any other topology that makes it a topological evs. Let $G\in\tau^\prime$. Since ($[0,\infty),\tau^\prime$) is a topological evs so by result \ref{ch4:r:3}, $G$ can be expressed as $G= \underset{x\in G}{\bigcup} (x+U_x)$ where, $U_x$ is an open, balanced and absorbing neighbourhood of $0$ in $\tau^\prime$, $\forall x \in G$. Now, by lemma \ref{ch4:l:4}, $U_x$ is an interval of the form $[0,a_x), \forall x \in G \Rightarrow x+U_x$ is also an interval of the form $[x,x+a_x), \forall x \in G.$ Now, we can combine intersecting intervals to express $G$ as :\\ \centerline
		{$G= \underset{j\in I}{\bigcup} H_j$ where,}\\ (i) $\{H_j : j\in I\}$ is a collection of pairwise disjoint intervals, $I$ being an index set and\\ (ii) for $H_i\neq H_j,$ no end point of $H_i$ can be an end point of $H_j$. \big[Here, (ii) ensures that the set of the form $[a,b)\cup [b,c)$ is taken as single interval [a,c) instead of taking union of two disjoint intervals $[a,b)$ and $[b,c)$ while constructing $H_i$'s.\big]

\underline{Claim:}  For any $j\in I$, if $0\in H_j$ then $H_j = [0,a)$ where, $a$ is a finite number or infinity; otherwise, $H_j$ is an open interval.

If not, let $H_j$ be a left closed interval having left end point $a>0$ for some $j\in I$. Now, $1.a=a\in G$ and `$\cdot$' is continuous so, $\exists \epsilon>0$ such that $B(1,\epsilon).a \subseteq G$. Since $H_i\cap H_j=\emptyset$ for $i\neq j$ and $a\in H_j$, we can choose $t\in B(1,\epsilon)$ very close to 1 with $| t|<1$ such that $t.a=| t|a \notin H_i$ for any $i\neq j$. Also, $| t|a<a$ $\Longrightarrow t.a \notin H_j$. Therefore $t.a \notin H_i$ for any $i\in I.$ i.e. $t.a\notin G$, which is a contradiction. Similarly, we can prove that $H_j$ can not be right closed. Hence, $H_j$ must be an open interval if $0\notin H_j$, $\forall j\in I.$  Thus, $G$ is union of intervals of the form $(a,b)$ or $[0,b)$ where, $a\in [0,\infty)$ and $b$ may be a finite number or $\infty \Longrightarrow$ $G \in \tau$. Now, $G\in\tau^\prime$ is arbitrary, so $\tau^\prime \subseteq \tau.$ Since $\tau^\prime$ is an arbitrary topology with respect to which $[0,\infty)$ forms a topological evs, so the usual topology is the finest topology that can make [0,$\infty)$ a topological evs.
\end{proof}

\end{document}